\documentclass[12pt,a4paper]
{article}
\usepackage{amsfonts}
\usepackage{amsmath}
\usepackage{mathrsfs}
\usepackage{graphicx}
\usepackage{amsmath,amsthm,amssymb,amscd}

\usepackage{color}
\usepackage{epsfig}

\newtheorem{Thm}{Theorem}[section]
\newtheorem{MainTheorem}{Theorem}
\newtheorem{Lem}[Thm]{Lemma}

\newtheorem{Def}[Thm] {Definition}

\newtheorem{Cor}[Thm]{Corollary}
\theoremstyle{remark}
\newtheorem{Rem} [Thm]{Remark}

\theoremstyle{claim}

\newtheorem{Que}[Thm]{Question}

\DeclareMathOperator{\Diff}{Diff}

\def\Leb{{\mbox{Leb}}}
\def\dim{{\mbox{\footnotesize{dim}}}}
\def\topo{{\mbox{\footnotesize{top}}}}

\begin{document}

\begin{center}
{\Large \bf   Dominated Splitting, Partial Hyperbolicity \\ \vspace{.3cm} and Positive Entropy }\\
\end{center}


\smallskip
\begin{center}
Eleonora Catsigeras$^*$
\end{center}
\begin{center}
Instituto de Matem\'{a}tica y Estad\'{\i}stica \lq\lq Rafael Laguardia\rq\rq (IMERL), Facultad de Ingenier\'{\i}a,
Universidad de la Rep\'{u}blica. Uruguay\\
\end{center}
\begin{center}
E-mail: eleonora@fing.edu.uy
\end{center}

\begin{center}
Xueting Tian$^{**}$
\end{center}
\begin{center}  School of Mathematical Sciences, Fudan University, Shanghai 200433, People's Republic of China \\
\end{center}
\begin{center}
E-mail: xuetingtian@fudan.edu.cn; tianxt@amss.ac.cn
\end{center}
\smallskip

\bigskip

\footnotetext {$^*$E. Catsigeras is partially supported by
Agencia Nacional de Investigaci\'{o}n e Innovaci\'{o}n (ANII), Comisi\'{o}n Sectorial de Investigaci\'{o}n Cient\'{\i}fica (CSIC) of Universidad de la Rep\'{u}blica, and Proyecto L'Or\'{e}al-Unesco-Dicyt of Uruguay. }
\footnotetext {$^{**}$X. Tian
is supported by National Natural Science Foundation of China(No. 11301088). }


\begin{abstract}

{\large Let $f:M\rightarrow M$ be a $C^1$  diffeomorphism with a dominated splitting on a   compact Riemanian manifold $M$ without boundary.
    We state and prove several   sufficient conditions for the topological entropy of $f$ to be positive. The conditions deal with   the dynamical  behaviour of the (non-necessarily invariant) Lebesgue measure. In particular, if the Lebesgue measure is  $\delta$-recurrent  then the entropy of $f$ is positive.  We give  counterexamples showing that these sufficient conditions are not necessary. Finally, in the case of partially hyperbolic diffeomorphisms,  we give a positive lower bound for the entropy   relating it with the dimension of the unstable and stable sub-bundles.}

\end{abstract}

\newpage

\section{Introduction}

Let $M$ be a compact boundary-less and connected manifold of finite dimension. Denote by $\mbox{Diff}^1(M)$ the space of $C^1$-diffeomorphisms  $f: M \mapsto M$. It is known that every Anosov system $f \in \mbox{Diff}^1(M)$ (or more generally, any horseshoe in a forward invariant open  set of $M$)  has positive entropy. Besides, due to the structural stability of Anosov diffeomorphisms,  any system $g \in \mbox{Diff}^1(M)$ close to $f$ is topologically conjugated to $f$. Thus,   the entropy function $$ h_{\topo}(\cdot):\Diff^1(M)\rightarrow\mathbb{R},\,f\mapsto h_{\topo}(f),$$ restricted to Anosov systems, is locally positively constant.

The positive entropy of Anosov systems is mainly obtained from its uniformly    hyperbolic behaviour.  But, since the uniform hyperbolicity is not a dense property in the whole space of differentiable dynamical systems,  researchers started to study other systems with some types of {\it weak hyperbolic properties}, such as nonuniform hyperbolicity,
 partial hyperbolicity and dominated splitting.

 On the one hand, it is  known that  non-uniformly hyperbolic system having at least one   positive Lyapunov exponent, have positive topological entropy.  Precisely, if $f$ is $C^{1+\alpha}$ and preserves a non-atomic ergodic hyperbolic measure, the classical $C^{1+\alpha}$ Pesin theory allows to prove that there is horseshoe. So, $f$ has positive entropy.

 On the other hand, partially hyperbolic systems also have   positive entropy, after the  recent result in \cite{SSV}.

 To extend these known results, in this paper we study the entropy of  diffeomorphisms  with (uniform and global) dominated splitting. This class of systems, which we denote by $\mbox{Diff}^1_{DS}(M)$, includes but is not reduced to   partially hyperbolic diffeomorphisms. Since the partially hyperbolic systems have positive entropy, the following question naturally arises:

 \begin{Que}\label{Que}

Has any $f$  in \em $\mbox{Diff}^1_{DS}(M)$ \em positive topological entropy?


\end{Que}


The answer is negative. In fact,   Gourmelon and Potrie \cite{Gourmelon-Potrie} have recently constructed a zero-entropy diffeomorphism on the torus $\mathbb{T}^2$ with dominated splitting.  For a seek of completeness we include this example in   Subsection \ref{subsectionExampleGourmelon-Potrie}.

After the negative answer of Question \ref{Que},  we
focus   on the search of  conditions for diffeomorphisms with dominated splitting such that:

 a) They include a much more general subfamily of diffeomorphisms in $\mbox{Diff}^1_{DS}$ than the partially hyperbolic ones.

 b) They imply   $h_{\topo}(f) >0$.

\vspace{.3cm}

 Along this paper we will state and prove several theorems that give  such kind of sufficient conditions. Our   results are based in the study of the dynamical behaviour of the (not necessarily invariant) volume measure on the manifold (the Lebesgue measure).

 In Section \ref{sectionStatements} we state the definitions and the main new results to be proved (Theorems \ref{Theorem1} to \ref{Theorem4}).
 Theorem 1 states that if the measurable sets with large Lebesgue measure have certain property of recurrency, then the entropy of the diffeomorphism with dominated splitting is positive. Theorem 2  assumes conditions on the so  called \em essential Lambda exponents. \em    One of these numbers is  the essential supremum w.r.t. Lebesgue  of the sum of Lyapunov-like exponents, which are defined for all the points $x \in M$. If the system has a dominated splitting and the essential Lambda-exponent  (which may be negative) is not very small, then the entropy of $f$ is positive. Theorems 3 and 4 hold  for particular cases: diffeomorphisms that preserve a smooth measure, and partially hyperbolic systems, respectively.

  Also in Section \ref{sectionStatements} we state and prove the immediate corollaries that are obtained from the  four main theorems. In Sections \ref{sect-3}, \ref{sect-4}, \ref{sectionProofSmooth} and \ref{sectionFinal} we prove the four main theorems.  
  Finally, in the Appendix (Section \ref{sectionExamples}) we provide   examples to prove that the answer to Question \ref{Que} is negative, and to show   that the converse statement  of Theorems \ref{Theorem1} is false.

\section{Definitions and statement of the results.}

 \label{sectionStatements}

  Before stating the main results  let us
recall the following definitions:

\begin{Def}\label{def:dominated}{\bf (Dominated Splitting)}  \em Let $f:M\rightarrow M$ be a $C^1$ diffeomorphism on a compact and connected Riemannian manifold $M$ without boundary. Let $TM=E\oplus F$ be a  $Df$-invariant   and continuous splitting, which is defined in all the points of the tangent bundle, such that $dim(E)\cdot dim(F)\neq 0.$

We call $TM=E\oplus F$   a {\it $\sigma-$dominated splitting} (where $E$ is the dominated sub-bundle and $F$ is the dominating sub-bundle), if there exists $\sigma>1$ such that $$\frac{\|Df|_{E(x)}\|}{m(Df|_{F(x)})}\leq \sigma^{-1}, \forall x\in M,$$ where for any linear transformation $A$ we denote  \begin{equation} \label{eqnMinimumNorm} m(A) := \displaystyle{\min_{\|u\|= 1} \|A \, u \|}.\end{equation}

\end{Def}

\begin{Rem}
 \label{Rem001}
 In Definition \ref{def:dominated}, the continuity of the splitting is redundant: it can be deduced from its $Df$-invariance and from the $\sigma-$ domination inequality (see for example \cite{BDV}). Since the manifold $M$ is assumed to be connected, the dimensions of the sub-bundles $E$ and $F$ are constant.

 \end{Rem}

\begin{Rem}
 \label{Rem002}
From the $\sigma-$ domination inequality of Definition \ref{def:dominated} we obtain:
$$\frac{\|Df^k|_{E(x)}\|}{m(Df^k|_{F(x)})}\leq
\prod_{i=0}^{k-1}\frac{\|Df|_{E(f^i(x))}\|}{m(Df|_{F(f^i(x))})}\leq \sigma^{-k}< \sigma^{-1}  \ \ \forall x\in M, \ \ \forall \ k \geq 1.$$ This means that, if  $T_MM=E\oplus F$ is a  $\sigma-$dominated splitting of $f$, then $T_MM=E\oplus F$ is also a  $\sigma-$dominated splitting of  $f^k$  for any integer number $k \geq 1$.
\end{Rem}

\begin{Rem}
 \label{Rem003}

The following is an equivalent definition of $\sigma-$dominated splitting  for some $\sigma >1$:
$TM=E\oplus F$ is a {\it dominated splitting} if there exists $C>0$ and $0<\lambda<1$ such that $$\frac{\|Df^n|_{E(x)}\|}{m(Df^n|_{F(x)})}\leq C \lambda^n  \ \ \forall \ x\in M, \ \ \forall \  n\geq 1.$$ In fact,  Gourmelon (\cite{Gour}) has proved that if the last inequality holds, then there  exists an adapted Riemannian
metric in the manifold $M$ for which $C=1$. So Definition \ref{def:dominated} holds.
\end{Rem}

\begin{Def}
\label{definitionRecurrentMeasures} \em {\bf (Measurable recurrence)} \em

Let $f: M \mapsto M$ be an homeomorphism.
We call a measurable set $B \subset M$ \em   recurrent to the future if    there exists $n_j \rightarrow + \infty$ such that $f^{n_j}(B) \cap B \neq \emptyset$ for all $j \geq 0$.

Let $\rho$ be a (non necessarily $f$-invariant) probability measure on $M$ and let $\delta$ be a real number such that $0 <\delta < 1$. We say that  the measure $\rho$ is  \em   $\delta$-recurrent by $f$, \em if any measurable set $B \subset M$ such that      $\rho(B) > 1 - \delta$ is recurrent.

\vspace{.3cm}

We say that the measure $\rho$ is \em recurrent by $f$ \em if it is $\delta$-recurrent for all $\delta$  such that $0 <\delta < 1$.

 It is immediate to check that if $\rho$ is $f$-invariant then $\rho$ is recurrent. In fact, if $B$ is a measurable set such that $\rho(B)>0$, applying Poincar\'{e} Recurrence Lemma we obtain  $f^n(B) \cap B \neq \emptyset $ for arbitrarily large values of $n $.

\end{Def}

Now we are ready to state our first main theorem:


\begin{MainTheorem}\label{Dom-MainTheorem0} \label{Theorem1}
 Let $f:M\rightarrow M$ be a $C^1$  diffeomorphism on a   compact Riemanian manifold $M$ exhibiting  a  $\sigma-$dominated splitting  $T M=E\oplus F$
  \em (with $\sigma>1$). \em Assume that   the Lebesgue  measure on $M$ is $\delta$-recurrent for $f$ for some $0 < \delta < 1$. Then the topological entropy of $f$ is positive.
   \end{MainTheorem}

We will prove Theorem \ref{Dom-MainTheorem0} in Section \ref{sect-3}. We remark that the hypothesis of   Theorem \ref{Dom-MainTheorem0},  which assumes that the Lebesgue measure is $\delta$-recurrent for some $0 < \delta < 1$,  is not necessarily satisfied for all the diffeomorphisms that have a dominated splitting and positive topological entropy. In fact, in Subsection \ref{sectionExample} we provide  an example that shows that the converse of Theorem \ref{Dom-MainTheorem0} is false.

\bigskip

Before stating our second main theorem, we need the following definition:

\begin{Def}
\label{DefinitionLambdaEssential} {\bf (Essential Lambda-Exponents)} \em
For any $Df$-invariant continuous sub-bundle $G$, and for any $x \in M$, define  the    real numbers $\lambda^{G,f}(x) $ and $\lambda^{G,f^-1}(x)$   by the following equalities: $$\lambda^{G,f}(x): = \limsup_{n \rightarrow + \infty} \frac{1}{n} \log |\mbox{det}(Df^{n}_x|_{G(x)})|, \ \ \ \ \ \lambda^{G,f^{-1}}(x): = \limsup_{n \rightarrow + \infty} \frac{1}{n} \log |\mbox{det}(Df^{-n}_x|_{G(x)})|. $$
 We call the following real numbers   \em essential Lambda-exponents   along $G$, \em to the future and the past respectively: $$\lambda_{ess}^{G,f}:= \mbox{Leb-ess sup } \lambda^{G,f}(x) , \ \ \ \ \lambda_{ess}^{G,f^{-1}}:= \mbox{Leb-ess sup } \lambda^{G,f^{-1}}(x), $$ where $\mbox{Leb-ess sup }  \psi(x)$ denotes the essential supremum, with respect to the Lebesgue measure, of the measurable  real function $\psi$.

\end{Def}

\begin{MainTheorem}
\label{Dom-MainTheorem} \label{Theorem2}
Let $f:M\rightarrow M$ be a $C^1$  diffeomorphism on a   compact Riemanian manifold $M$ exhibiting  a  $\sigma-$dominated splitting  $T M=E\oplus F$ for $f$  \em (with $\sigma>1$). \em
If at least one of the following inequalities holds: \em
  \begin{eqnarray}
  \label{equation41}
 \lambda_{ess}^{TM,f} \ \ &>& -\mbox{ dim} (E) \log \sigma \\
 \label{equation42}
  \mbox{ or }  \ \ \ \   \lambda_{ess}^{TM,f^{-1}} &> & -\mbox{ dim} (F) \log \sigma  \\
 \label{equation43}
 \mbox{ or } \ \ \ \  \lambda_{ess}^{F,f} \ \ \ \ \   &>& \ \ \ 0  \\ \label{equation43b}
 \mbox{ or } \ \ \ \   \lambda_{ess}^{E,f^{-1}}\ \    &>&\ \ \ 0,
  \end{eqnarray}
\em  then     the topological entropy of
   $f$ is positive.
\end{MainTheorem}

We will prove Theorem \ref{Dom-MainTheorem} in Section \ref{sect-3}.

\begin{Cor}
\label{Cor-Dom-MainTheorem}
Let $f:M\rightarrow M$ be a $C^1$  diffeomorphism on a   compact Riemanian manifold $M$ exhibiting  a  $\sigma-$dominated splitting  $T M=E\oplus F$ for $f$  \em (with $\sigma>1$). \em Assume that   at least one of the following inequalities holds: \em
  \begin{eqnarray}
 \label{integrable-equation41}
 \ \ \ \ \ \ \ \ \  \limsup_{n \rightarrow + \infty} \frac{1}{n} \int\log |\mbox{det}(Df^{n}  (x))| d Leb \  & > & -\mbox{ dim} (E) \log \sigma \\  \mbox{ or } \ \ \ \
 \label{integrable-equation42}\limsup_{n \rightarrow + \infty} \frac{1}{n}\int \log |\mbox{det}(Df^{-n}  (x))| d Leb & > & -\mbox{ dim} (F) \log \sigma \\  \mbox{ or } \ \ \ \  \label{integrable-equation43}\limsup_{n \rightarrow + \infty} \frac{1}{n}\int \log |\mbox{det}(Df^{n}|_{F_x})|  d Leb \ \ \; &> &\ \ \  0 \\   \mbox{ or } \ \ \ \ \label{integrable-equation43b}
  \limsup_{n \rightarrow + \infty} \frac{1}{n} \int \log |\mbox{det}(Df^{-n}|_{E_x})| d Leb \ &> & \ \ \  0.\end{eqnarray}
 \em Then, the topological entropy of $f$ is positive.
\end{Cor}

\begin{proof}
The statement of Corollary \ref{Cor-Dom-MainTheorem} is an immediate consequence of   Theorem \ref{Theorem2}, taking into account   Definition \ref{DefinitionLambdaEssential} and applying   Fatou Lemma.
\end{proof}

\subsection {The smooth-invariant measure case} \label{subsectionStatementSmooth}

 In the particular case that   $f$ preserves a smooth measure $\mu$ (i.e., $\mu$ is absolutely continuous w.r.t. Lebesgue measure),   we obtain the following result, which is indeed an immediate corollary of Theorem \ref{Theorem1}:

 \begin{Cor}\label{Dom-MainTheorem-SmoothMeasure} \label{Corollary2.10}
 If $ f\in  \Diff_{DS}(M)$   preserves a smooth measure, then its topological entropy  is positive.

\end{Cor}

  \begin{proof}
  Denote by $\Leb$ the Lebesgue probability measure on $M$. If $\mu $ is $f$-invariant, then $\mu$ is $f$-recurrent (for any $0 < \delta < 1$)  due to Poincar\'{e} Recurrence Lemma. Besides $\mu \ll \Leb$, and so, the measurable set $B$ where the density of $\mu$ is positive satisfies $\mu(B) = 1$. Denote $\alpha: =\Leb(B) >0$. Any measurable set $A$ such that $\Leb (A) > 1- \alpha$   intersects $B$ on a set $A \cap B$ with positive Lebesgue measure. Besides the density of $\mu$ at any point $x \in A \cap B$ is positive. Then $\mu(A \cap B) >0$. We deduce that $A \cap B$ is an $f$-recurrent set (because $\mu$ is $f$ invariant). So, $A$ is also a recurrent set. We have proved that any measurable set $A$ such that $\Leb(A) > 1 - \alpha$ is recurrent. From Definition \ref{definitionRecurrentMeasures}, $\Leb$ is   an  $\alpha$-recurrent measure.
   Finally, we apply Theorem \ref{Dom-MainTheorem0} to conclude that  $h_{\topo}(f) >0$. \end{proof}

    The proof of   Corollary \ref{Dom-MainTheorem-SmoothMeasure} can be also easily and independently deduced   from the following already known theorem:

\vspace{.3cm}

\noindent{\bf Theorem (Pesin-like formula, Sun-T. \cite{SunTian})}

\em If $f \in \Diff^1(M)$ has a $\sigma$-dominated splitting ($\sigma >1$) $TM = E \oplus F$, and if $\mu$ is a smooth $f$-invariant probability measure, then: \em
\begin{equation}
\label{eqnSunTian}
h_{\mu}(f) \geq \int  \log |\det Df|_F| \, d \mu = \int \sum_{i= 1}^{\mbox{\dim}(F)} \chi_{i} \, d \mu  ,\end{equation} \em
where $h_{\mu}(f)$ denotes the metric entropy of $f$ w.r.t the measure $\mu$ and \em $$\chi_1 \geq \chi_2 \geq \ldots \geq \chi_{\mbox{\footnotesize dim}(M)}$$ \em are the Lyapunov exponents  defined $\mu$-a.e. \em

\vspace{.3cm}

 \noindent {\bf Independent proof of Corollary \ref{Dom-MainTheorem-SmoothMeasure} deduced from Theorem Sun-T.}

 \vspace{.2cm}

 We include here   a different proof of Corollary \ref{Dom-MainTheorem-SmoothMeasure}, which is independent of Theorems \ref{Theorem1} and \ref{Theorem2}, because some of its arguments   will be useful to obtain further results.

 \begin{proof}
 Since $f^{-1}$ has also a dominated splitting, and $\mu \ll \mbox{Leb.}$ is also $f^{-1}$-invariant, we can apply Inequality (\ref{eqnSunTian}) to $f^{-1}$:
\begin{equation}
\label{eqnSunTian2}h_{\mu}(f) = h_{\mu}(f^{-1}) \geq -\int \sum_{j= \mbox{\footnotesize dim}(F) + 1}^{\mbox{\footnotesize dim}(M)} \chi_{j} \, d \mu \ \ \ \mbox{ if } \mu \ll \mbox{Leb.} \end{equation}
Either $\displaystyle \int \sum_{j= \mbox{\footnotesize dim}(F) + 1}^{\mbox{\footnotesize dim}(M)} \chi_{j} \, d \mu >0$, and so by (\ref{eqnSunTian}) the entropy is positive, or  $\displaystyle \int \sum_{i= 1}^{\mbox{\footnotesize dim}(F)} \chi_{i} \, d \mu \leq 0. $ So, it is enough to prove that the entropy is also positive under the assumption that  $\displaystyle \int \sum_{i= 1}^{\mbox{\footnotesize dim}(F)} \chi_{i} \, d \mu \leq 0. $
 From the dominated splitting condition we obtain $$\chi_i \geq \log\sigma + \chi_j \ \ \forall \ 1 \leq i \leq \mbox{dim}(F) < \mbox{dim}(F) + 1 \leq j \leq \mbox{dim}(M).$$ Thus, we can bound from  above the integral at right in Inequality (\ref{eqnSunTian2}) as follows:
$$\int \sum_{j= \mbox{\footnotesize dim}(F) + 1}^{\mbox{\footnotesize dim}(M)} \chi_{j} \, d \mu  \leq \int \mbox{dim}(E) \cdot \Big(\log \sigma^{-1} + \min_{1 \leq i \leq \mbox{\footnotesize dim} F} \chi_i \Big) \, d \mu \leq $$ $$   \mbox{dim}(E)  \cdot \Big (\log \sigma^{-1} + \frac{1}{\mbox{dim} (F)} \cdot \int \sum_{i= 1}^{\mbox{\footnotesize dim}(F)} \chi_{i} \, d \mu \Big) \leq   \mbox{dim}(E) \cdot \log \sigma^{-1} < 0.$$
  So, Inequality (\ref{eqnSunTian2}) gives $h_{\mu}(f) >0$, ending the proof of Theorem \ref{Dom-MainTheorem-SmoothMeasure}.
\end{proof}

\vspace{.3cm}

    \bigskip

    We point out that the latter proof is adaptable to   systems that preserve a smooth probability measure, and that have a non-uniform and non-global almost dominated splitting,   according to the following definition:

    \begin{Def}
     \label{DefinitionAlmostDom} {\bf (Almost dominated splitting)} \em
       Fix a point $x\in M$ and denote its orbit $\{f^n(x)\}_{n \in \mathbb{Z}}$ by $orb(x)$. A splitting $$T_{orb(x)}M=E_{orb(x)}\oplus F_{orb(x)}$$ is called
\emph{$N(x)$-dominated at point $x$}, if it is $Df$-invariant  and there exists a constant $N(x)\in
\mathbb{Z}^+$ such that
$$\frac
{\|Df^{N(x)}|_{E(f^{j}(x))}\|}{m(Df^{N(x)}|_{F(f^{j}(x))})}\leq\frac12,\,\,\forall
\,j\in\mathbb{Z}.$$
Let $\mu$ be an $f-$invariant measure $\mu$ and let $N(\cdot):M\rightarrow
\mathbb{N}$ be an $f$-invariant measurable function. We say  $\mu$ has an {\it   almost dominated splitting,}  if for
$\mu\,a.\,\,e.\,\,x\in M,$ there is an $N(x)$-dominated splitting
$$T_{orb(x)}M=E_{orb(x)}\oplus F_{orb(x)}$$ at $x$. We say  $\mu$ has a {\it non-trivial almost dominated splitting,} if it has an almost dominated splitting and the set for which the following inequality holds has $\mu-$positive measure: $$\mbox{dim}(E(x)) \cdot \mbox{dim}(F(x))\neq 0.$$

Notice that if $\mu$ if $f$-invariant and has an almost dominated splitting for $f$, then it has an almost dominated splitting for $f^{-1}$.
 \end{Def}
Now we state the  main new result in the case of smooth invariant measure:

\begin{MainTheorem} \label{Theorem3}
\label{Smooth-Measure-Positive-Entropy} Let $f \in \mbox{Diff}^1(M)$   preserving a smooth measure $\mu$. Assume that $\mu$   has a non-trivial almost dominated splitting. Then $\mu$ has positive metric entropy, hence $f$  has positive topological entropy.

\end{MainTheorem}

We will prove Theorem \ref{Smooth-Measure-Positive-Entropy} in Section \ref{sectionProofSmooth}.
In particular, Theorem \ref{Smooth-Measure-Positive-Entropy} immediately implies the following corollary for volume-preserving diffeomorphisms. Let \em $\Leb$ \em denote the Lebesgue measure on $M$ (i.e. the volume measure), and let $\Diff_{Leb}^1(M)$ denote the space of all $C^1$ volume-preserving diffeomorphisms.

\begin{Cor}\label{Volume-Measure-Positive-Entropy} \label{corollary2.13} If $f\in \Diff_{Leb}^1(M) $   and    $\Leb$   has a non-trivial almost dominated splitting, then  $f$  has positive  entropy.

\end{Cor}

Now we joint Corollary \ref{corollary2.13} with the following known result:

\vspace{.3cm}

\noindent{\bf Theorem (Bochi-Viana  \cite{BV})}
\em There is a residual subset $\mathcal {R}\subseteq \Diff_{Leb}^1(M)$ such
that for every $f\in\mathcal {R}$ and for ${Leb}-$a.e.$x\in M$
the Oseledec splitting of $f$ is either trivial (i.e. all Lyapunov
exponents are zero) or  dominated at $x$. \em

\vspace{.3cm}

As a consequence of Theorem of Bochi-Viana and Corollary \ref{Volume-Measure-Positive-Entropy} one immediately obtains:
\begin{Cor}\label{Generic-Volume-Measure-Positive-Entropy}
There is a residual subset $\mathcal {R}\subseteq \Diff_{Leb}^1(M)$ such
that for every $f\in\mathcal {R}$, either for $\Leb -$a.e.$x\in M$
  all Lyapunov
exponents are zero, or  $\Leb$ has positive entropy and thus $f$  has positive entropy.

\end{Cor}

It is   known (see \cite{Yang}) that  for any  $C^1$ diffeomorphism
$f$ far away from homoclinic tangencies and for any $f$-ergodic measure
$\nu$, the stable, center and unstable bundles of the
Oseledec splitting are dominated on $\textrm{supp}(\nu)$  (the support of $\nu$),  and besides   the center bundle is at most one dimensional.
 Then, one can use the Ergodic Decomposition Theorem (see for instance  \cite{Walter}) to obtain that for any $f$-invariant measure $\mu$, $\mu-$a.e. $x$ has an Oseledec splitting that is dominated  at $x$.

So, from Theorem \ref{Theorem3}, we deduce:

\begin{Cor}\label{PesFormula-Thm:inverse} Let $M$ be a Riemannian compact manifold with $dim(M)\geq 2$. Let $f: M\to M$ be a $C^1$ diffeomorphism  far from tangencies  that   preserves  a smooth invariant
measure $\mu$. Then $f$ has positive entropy.
\end{Cor}

\noindent In particular Corollary \ref{PesFormula-Thm:inverse} holds for volume preserving diffeomorphisms far from tangencies.

 \subsection{The partially hyperbolic case} \label{subsectionPartialHyp}

\bigskip

The  definition of (strong) partial hyperbolicity requires the existence of a continuous splitting in three $Df$-invariant sub-bundles, such that one (which is called the unstable bundle) is uniformly expanding, other one (which is called the stable bundle) is uniformly contracting, and the third one (which is called the center bundle)   is dominated by the unstable bundle and dominates the stable one.

Here we adopt a  more general notion of partial hyperbolicity  by using a splitting into two sub-bundles:

 \begin{Def}
 \label{DefinitionPartialHyperbolicity} {\bf (Partial hyperbolicity)}  \em
 We call $TM=E\oplus F$    a {\it $Df$-partially hyperbolic splitting,}  if it is a dominated splitting   such that either the dominated sub-bundle $E$ is uniformly contracting by $Df$, or the dominating sub-bundle $F$ is uniformly expanding by $Df$. Precisely, besides the domination inequality of Definition \ref{def:dominated}, there exists $C>0$ and $ \alpha>1$ such that either \begin{equation}
  \label{eqn100a}
  {\|Df^n_x|_{E(x)}\|} \leq C \alpha^{-n}, \forall x\in M,\,\, \forall \, n\geq 1,\end{equation}  or   \begin{equation}
  \label{eqn100b} {\|Df^{-n}_x|_{F(x)}\|} \leq C \alpha^{-n}, \forall x\in M,\,\, \forall \, n\geq 1.\end{equation}

 A diffeomorphism $f: M \mapsto M $ is called \em partially hyperbolic \em if the tangent bundle has a  $Df$-partially hyperbolic splitting.
   \end{Def}

  According to Definition \ref{DefinitionPartialHyperbolicity}, Anosov diffeomorphisms  for instance,  are particular cases of partially hyperbolic  diffeomorphisms,  and these latter are particular cases of diffeomorphisms with (global and uniform) dominated splitting.

\begin{Thm}\label{PartialHyper-MainTheorem} {\bf (Saghin-Sun-Vargas) \cite{SSV}}

If $f \in \mbox{Diff}^1(M)$ is partially hyperbolic then the topological entropy of $f$ is positive.

\end{Thm}

Let us see that the Theorem of Saghin-Sun-Vargas can be   proved also as a particular case of Theorem \ref{Theorem2}:

\vspace{.3cm}

\noindent{\bf Proof of Theorem \ref{PartialHyper-MainTheorem} as a corollary of Theorem \ref{Theorem2}.}

\begin{proof}

 On the one hand, if inequality (\ref{eqn100a}) holds, then:
$$|\det Df^{n}_{}|_{F(x)}| =  {|\det Df^{-n}_{ f^{n}(x) }|_{F(f^{n}(x))} |}^{-1}  \geq {\|Df^{-n}_{f^{n}(x)}|_{F(f^{n}(x))} \|}^{-1}  \geq C^{-1} {\alpha^n}  \ \ \ \forall \ x \in M. $$
 From the inequality above and Definition \ref{DefinitionLambdaEssential}, we obtain:
$$\lambda^{F,f}(x) = \limsup_{n \rightarrow + \infty} \frac{1}{n} \log |\det Df^{n}_{}|_{F(x)}| \geq \lim_{n \rightarrow + \infty} \frac{1}{n} (\log \alpha^n - \log C) = \log \alpha >0 \ \ \forall \ x \in M.$$
So $\lambda_{ess}^{F,f} \geq \log \alpha >0$, and by Theorem \ref{Theorem2} the entropy of $f$ is positive.
On the other hand, if inequality (\ref{eqn100b}) holds, the latter argument works with $f^{-1}$  instead of $f$ and the sub-bundle $E$ instead of $F$. So $\lambda_{ess}^{E, f^{-1}} \geq \log \alpha >0$, and by Theorem \ref{Theorem2} the entropy of $f$ is positive.  \end{proof}

\begin{Rem}\label{Rem-Dom-Proof}
We have shown that the new result stated on Theorem \ref{Dom-MainTheorem}  is a generalization of Theorem of  Saghin-Sun-Vargas  firstly proved in \cite{SSV}. The authors of \cite{SSV}   constructed a $n$-separated set on the unstable (or stable) manifold, and proved, using the uniformly exponential contraction along $E$, or the uniformly exponential expansion along $F$, that the cardinality of this $n$-separated set has positive exponential growth with $n$. This method does not work in the  general case of dominated splitting (without partial hyperbolicity)  because  there may not exist global uniform contraction or expansion in  the sub-bundles of the dominated splitting.  For this reason the proof of   Theorem \ref{Dom-MainTheorem} must take a different route than the first proof of Theorem \ref{PartialHyper-MainTheorem} in \cite{SSV}. We mainly base the proof of Theorem \ref{Theorem2} on some recent advances on Pesin's entropy formula for the so called \em SRB-like measures \em of $C^1$ diffeomorphisms with dominated splitting (\cite{CCE}).

\end{Rem}

The following Theorem \ref{Theorem4}  strengthens the Theorem of Saghin-Sun-Vargas \cite{SSV}. In fact, in Theorem \ref{Theorem4} we will provide an explicit positive lower bound $k$ of the topological entropy, and also an explicit description of a set of $f$-invariant probability measures whose metric entropies are lower bounded by $k$.

Before stating Theorem \ref{Theorem4},  we  recall the following definition, which was taken from \cite{CE}:

\begin{Def} {\bf The omega-limit  set  in the space of probabilities.} \em

Denote by ${\mathcal P}$ the space of probability Borel-measures on the manifold $M$, endowed with the weak$^*$ topology. Denote by ${\mathcal P}_f$ the set of $f$-invariant measures in ${\mathcal P}$. For any point $x \in M$, denote by $\delta_x$ the Dirac-probability measure supported on $\{x\}$. Construct the set
$$p\omega(x, f) := \Big\{\mu \in {\mathcal P}: \  \lim_{j \rightarrow + \infty} \frac{1}{n_j} \sum_{i=1}^{n_j - 1} \delta_{f^i(x)} = \mu \mbox{ for some sequence } \ n_j \rightarrow + \infty  \Big\}. $$
We call $p\omega(x, f)$ the \em   limit set in the space of probabilities \em of the future orbit  of $x$ by $f$.

It is standard to check that for all $x \in M$  the set $p\omega(x)$ is nonempty, weak$^*$-compact and contained in ${\mathcal P}_f$. Consider  also the nonempty weak$^*$-compact set $p \omega(x, f^{-1}) \subset {\mathcal P}_f$.
 We call it the \em   limit set in the space of probabilities \em of the past orbit  of $x$.

 \label{definitionpomega(x)}

\end{Def}

\begin{MainTheorem}\label{PartialHyper-SRB-like} \label{Theorem4}
Let $f:M\rightarrow M$ be a $C^1$ diffeomorphism on a   compact Riemanian manifold $M$ with a dominated splitting $TM=E\oplus F$.

 \noindent {\bf (a) }
If there exist $C>0$ and $\alpha > 1$ such that $ {\|Df^{-n}|_{F(x)}\|} \leq C \alpha^{-n}, \forall x\in M,\,\, n\geq 1,$   then for Lebesgue-almost all $x \in M$ and for all $\mu \in p\omega(x, f)$:
 $$  h_\mu(f)\geq \int \log|\det Df|_F|\, d \mu \geq \mbox{\em dim}(F)\log \alpha >0.$$

 \noindent {\bf (b) } If there exist $C>0$ and $\alpha > 1$ such that $ {\|Df^{n}|_{E(x)}\|} \leq C \alpha^{-n}, \forall x\in M,\,\, n\geq 1,$   then for Lebesgue-almost all $x \in M$ and for all $\mu \in p\omega(x, f^{-1})$:
 $$  h_\mu(f)\geq \int \log|\det Df^{-1}|_E|\, d \mu \geq \mbox{\em dim}(E)\log \alpha >0.$$

\end{MainTheorem}

We will prove Theorem \ref{PartialHyper-SRB-like} in Section \ref{sectionFinal}. As a consequence of Theorem \ref{PartialHyper-SRB-like}, we can strengthen Theorem \ref{PartialHyper-MainTheorem} as follows.

\begin{Thm}\label{Strong-PartialHyper-MainTheorem}

If $f \in \mbox{Diff}^1(M)$ is partially hyperbolic then there is a real number $t>0$ and a neighborhood $\cal U$ of $f$ such that the topological entropy of each $g\in \cal U$ is larger or equal to $t$.

\end{Thm}

{\bf Proof.} We just consider the case (a) in Theorem \ref{PartialHyper-SRB-like} and another case is similar. Take $\epsilon>0$ such that $\alpha-\epsilon>1.$ Since partial hyperbolicity is an open property, then there is a neighborhood $\cal U$ of $f$ such that   each $g\in \cal U$ satisfies that: there is a dominated splitting $TM=E_g\oplus F_g$ (called continuation of $E_f\oplus F_f=E\oplus F$)   such that $ {\|Dg^{-n}|_{F_g(x)}\|} \leq C (\alpha-\epsilon)^{-n}, \forall x\in M,\,\, n\geq 1.$ By  Theorem \ref{PartialHyper-SRB-like}, the topological entropy of each $g\in \cal U$ is larger or equal to $ \mbox{\em dim}(F_g)\log (\alpha-\epsilon)= \mbox{\em dim}(F)\log (\alpha-\epsilon).$ Take $t:=\mbox{\em dim}(F)\log (\alpha-\epsilon)$ and we complete the proof. \qed

 \section{Proof  of Theorem  \ref{Theorem1} as  a corollary of  Theorem \ref{Theorem2}.} \label{sect-3}

 To start the proofs, we will first  show that Theorem \ref{Dom-MainTheorem0} is indeed a corollary of Theorem \ref{Dom-MainTheorem}. Recall Definition \ref{DefinitionLambdaEssential}   and consider inequalities (\ref{equation41})   and   (\ref{equation42})     of Theorem   \ref{Dom-MainTheorem}. If we prove  that at least one of the exponents $\lambda _{ess}^{TM, \ f}$,  $\lambda_{ess}^{TM, \ f^{-1}}$ is not negative, then at least one of the  inequalities (\ref{equation41})   and   (\ref{equation42}) will hold. So, from Theorem \ref{Theorem2} we will deduce  that the entropy of $f$ is positive. This latter is the route of the proof of Theorem \ref{Theorem1} as a corollary of Theorem \ref{Theorem2}. Later, we will prove Theorem \ref{Theorem2} independently.

 \begin{Lem}
 \label{lemmaAddedByEleonora}

 Let $f \in \mbox{Diff}_{DS}^1(M)$. If the Lebesgue measure is $\delta$-recurrent for some $0 < \delta < 1$, then  either $\lambda _{ess}^{TM, \ f} \geq 0$ or  $\lambda_{ess}^{TM, \ f^{-1}} \geq 0$.

 \end{Lem}

 \begin{proof}   Assume by contradiction that there exists $a >0$ such that $\lambda ^{TM, \ f}(x) < - a$ and $  \lambda ^{TM, \ f^{-1}}(x) < -a$  for Lebesgue almost all $x \in M$. In other words, the following two inequalities hold simultaneously Leb-a.e. $x \in M$:
 \begin{equation}
 \label{eqn01}
 \limsup_{n \rightarrow + \infty} \frac{1}{n} \log |\mbox{det}Df^n(x) | < -a, \ \ \ \ \ \limsup_{n \rightarrow + \infty} \frac{1}{n} \log |\mbox{det}Df^{-n}(x) | < -a.
 \end{equation}
 Thus, for Lebesgue a.e. $x \in M$ there exists  a (minimum)  natural number $N= N(x)$ such that
 $$|\mbox{det}Df^n(x) | < e^{-na} \mbox{ and } |\mbox{det}Df^{-n}(x) | < e^{-na}  \ \ \forall \ n \geq N(x).$$ For each $N \in \mathbb{N}$ construct $C_N :=\{x \in M: \ \ N(x) \leq N\}$.   Since $C_N \subseteq C_{N+1}$ and $\Leb\Big(\bigcup_{N= 1}^{+ \infty} C_N\Big) = 1$ we have
 $\lim_{N\rightarrow + \infty} \Leb(C_N) = 1$.
 Fix   $N$ such that $$\Leb (C_N) > 1 - \delta.$$
Fix $n \geq N$. We have $|\mbox{det}(Df^n(x))| \leq \ e^{-n \alpha},   \ |\mbox{det}(Df^{-n}(x))|   \leq   e^{-n \alpha} $  for all $x \in C_N$, from where we deduce the following inequalities  for any measurable set $B$:
 \begin{equation} \label{eqn02} \Leb(C_N \cap f^n(B))\leq e^{- n \alpha} \Leb(B), \ \ \ \ \ \ \ \Leb(C_N \cap f^{-n}(B))  \leq e^{- n \alpha} \Leb(B). \end{equation}
 We put $B_1 := C_N  \cap f^{-n}(C_N)$ and $B_2:= C_N \cap f^n(C_N)$ instead of $B$  in the inequality at left and at right of (\ref{eqn02}) respectively. We obtain:
 $$\Leb(B_2) = \Leb(C_N \cap f^n(B_1)) \leq e^{-n \alpha} \Leb (B_1),$$
 $$\Leb (B_1) = \Leb (C_N \cap f^{-n}(B_2) \leq e^{-n \alpha}\Leb(B_2).$$
 Since $0 < e^{-n \alpha} < 1$, the above inequalities imply $0 =\Leb(B_1)= \Leb(B_2)$.
   So, we have proved that $$\Leb(C_N \cap f^n(C_N)) = 0 \  \ \forall \ n \geq N.$$
  Finally, construct
   $\displaystyle A_N = C_N \setminus \Big(\bigcup_{n= N}^{+ \infty} f^n(C_N)\Big)              $
  to conclude that
   $$\Leb (A_N) = \Leb(C_N) > 1 - \delta, \ \ \
   f^n(A_N) \cap A_N = \emptyset \ \ \forall \ n \geq N, $$
   which contradicts the   $\delta$-recurrence of the Lebesgue measure.  \end{proof}

 After Lemma \ref{lemmaAddedByEleonora},  to end the proof of Theorem \ref{Theorem1} it is now  enough to prove independently Theorem \ref{Dom-MainTheorem}.


\section{Independent proof   of Theorem  \ref{Theorem2}.} \label{sect-4}

\noindent {\em Route of the proof of Theorem}  \ref{Theorem2}:  We will use similar ideas to those in Subsection \ref{subsectionStatementSmooth} for $C^1$ diffeomorphisms with dominated splitting that preserve a smooth measure. But    now, smooth invariant measures may not exist. So, we do not have from the very beginning  an adequate invariant measure that satisfies simultaneously inequalities (\ref{eqnSunTian}) and (\ref{eqnSunTian2}).
Anyway,   we will   construct two  or more invariant   probabilities, some satisfying inequality (\ref{eqnSunTian}) and the other ones satisfying inequality (\ref{eqnSunTian2}). Finally we will prove that at least one of those measures has positive entropy.

The construction of such adequate probabilities will be based on the theory of SRB-like measures for $C^1$ maps  introduced in \cite{CE}. We will apply a result in \cite{CCE}, which provides a Pesin-like formula for the entropy  to all the SRB-like measures of any $f \in \mbox{Diff}^1_{DS}(M)$. This formula  was  previously proved in \cite{SunTian}   for the particular case of   $f$ preserving a smooth measure.

\vspace{.3cm}

  Recall Definition \ref{definitionpomega(x)} of the limit set $p\omega(x, f)$ in the space ${\mathcal P}$ of probabilities of the future orbit of $x$ by $f$. Fix a metric $\mbox{dist}$ in ${\mathcal P}$ that endows the weak$^*$-topology. Let us recall the definition of the SRB-like measures (taken  from \cite{CE}):

 \begin{Def} {\bf (SRB-like measures)} \em
  \label{definitionSRB-like}

 A probability measure $\mu\in \mathcal{P}_f$ is \em SRB-like \em   (or  observable or  pseudo-physical) if, for any $\epsilon >0$,  the set $$A_\varepsilon(\mu)=\{x\in M\colon \ \mbox{dist}(p\omega (x, f),\mu) <\varepsilon \}$$ has positive Lebesgue measure. The set $A_\varepsilon(\mu)$ is called basin of $\varepsilon-$attraction of $\mu.$

We denote by $\mathcal{O}_{f}$ the set of all the SRB-like measures for $f$.

\end{Def}


%
%
%
%
%
%
%
%
%
%
%
%
%
%
%
%

  We will use the following previous theorems from \cite{CE}, \cite{CCE} and \cite{SunTian}:

\begin{Thm}\label{SRB-like} {\bf (C.-Enrich \cite{CE})} \label{TheoremCE}

 For any continuous map $f : M \mapsto M$ the set $\mathcal{O}_{f} $ of   SRB-like measures is nonempty, weak$^*$-compact, and  contains $p\omega(x,f)  $ for a.e. $x \in M$.

\end{Thm}

\begin{Thm}\label{Dom-EntropyFormula} \label{theoremCCE}
{\bf (Pesin-like formula, C.-Cerminara-Enrich \cite{CCE})}

 If $f \in \mbox{Diff}^{1} (M)$ has a  dominated splitting $TM = E \oplus F$, and if $ \mu \in {\mathcal O}_f$ \em (i.e. $\mu $ is SRB-like), \em then
\em \begin{equation} \label{eqnDom-EntropyFormula0}
 h_\mu(f)\geq \int \sum_{i=1}^{\dim (F)}\chi_i(x)d\mu = \int \log |\det Df|_F| \, d \mu,\end{equation} \em where  \em $\chi_1 \geq\chi_2 \cdots \geq \chi_{\dim(M)} $ \em denote the Lyapunov exponents defined $\mu$-a.e.
 \end{Thm}

 \vspace{.3cm}

Now, we are ready to prove Theorem \ref{Theorem2}, and hence, also end  the proof of Theorem \ref{Theorem1}.

\vspace{.2cm}

\noindent {\em Proof of Theorem  } \ref{Theorem2}.
We will divide the proof into two parts.
In the first part we will prove that inequality (\ref{equation43}) implies $h_{\topo}(f) >0$.
In the second part we will prove that inequality (\ref{equation41}) also implies $h_{\topo}(f) >0$.
These two parts are enough to prove completely Theorem \ref{Theorem2} because they also hold for $f^{-1}, E$ and $ F$ in the roles of $f, F$ and $ E$  respectively.

\begin{proof}
 {\em of the 1st. part:}

 Assume inequality (\ref{equation43}). By Definition \ref{DefinitionLambdaEssential} the following set $B$ has positive Lebesgue measure:
 $$B:= \{x \in M: \limsup_{n \rightarrow + \infty} \frac{1}{n} \log  |\det Df^n_x|_{F(x)}| > 0\}.$$
 From Theorem \ref{TheoremCE}: $$p\omega(x,f) \subset {\mathcal O}_f \ \ \mbox{Leb.- a.e.} x \in M.$$   Choose and fix  a point $x \in B$ such that $p\omega(x,f) \subset {\mathcal O}_f$, and fix a sequence $n_j \rightarrow + \infty$ such that
 \begin{equation}\label{eqn105}\lim_{j \rightarrow + \infty} \frac{1}{n_j} \log  |\det Df^{n_j}_x|_{F(x)}| = r >0.\end{equation}
 Choose a subsequence of $\{n_j\}_j$ - which for simplicity we still denote by $\{n_j\}_j$ -  such that the following limit exists in the space ${\mathcal P} $ of probabilities endowed with the weak$^*$-topology (we call this limit $\mu$):
 \begin{equation}\label{eqn106}\lim_{j \rightarrow + \infty} \frac{1}{n_j} \sum_{i=0}^{n_j - 1} \delta_{f^i(x)} = \mu \in {\mathcal P}.\end{equation}
  After Definition \ref{definitionpomega(x)}, $\mu \in p\omega(x,f) \subset {\mathcal O}_f$. So, applying Theorem \ref{theoremCCE}:
  \begin{equation}\label{eqn107}h_{\mu}(f) \geq \int \psi \, d \mu, \mbox{ where } \psi:= \log |\det Df|_F|.\end{equation}
   By the definition of the weak$^*$ topology in ${\mathcal P}$ (since $\psi$ is a continuous real function), and   from equalities (\ref{eqn106}) and   (\ref{eqn105}), we deduce:
   \begin{equation}\label{eqn108}\int \psi \, d \mu = \lim _{j \rightarrow + \infty} \frac{1}{n_j} \sum_{i= 0}^{n_j - 1} \psi (f^i(x)) = \lim _{j \rightarrow + \infty} \frac{1}{n_j}   \sum_{i= 0}^{n_j - 1}  \log |\det Df_  {f^i(x)}|_{F(f^i(x))}|= $$ $$ \lim_{j \rightarrow + \infty}\frac{1}{n_j} \log |\det Df^{n_j}_x|_{F(x)} = r >0.  \end{equation}
 Joining inequalities (\ref{eqn107}) and (\ref{eqn108}) we conclude that $h _{\mu}(f) >0$, as wanted.
 \end{proof}

\begin{proof}
{\em of the 2nd. part:}  Assume that inequality (\ref{equation41}) holds. Arguing as in the first part we find a point $x \in M$, a sequence $n_j \rightarrow + \infty$ and an SRB-like   measure $\mu \in {\mathcal O}_f$ such that

\begin{equation}\label{eqn110}\lim_{j \rightarrow + \infty} \frac{1}{n_j} \log  |\det Df^{n_j}_x | = s >-\mbox{dim}(E) \log \sigma,\end{equation}
  \begin{equation}\label{eqn111}h_{\mu}(f) \geq \int \log |\det Df_F|  \, d \mu, \end{equation}
   \begin{equation}\label{eqn112}\int |\det Df  | \, d \mu = \lim _{j \rightarrow + \infty} \frac{1}{n_j} \sum_{i= 0}^{n_j - 1} |\det Df_{f^i(x)} |   =  \lim_{j \rightarrow + \infty}\frac{1}{n_j} \log |\det Df^{n_j}_x |  = s.  \end{equation}
  Since $E \oplus F = TM$ is a $Df$-invariant splitting and $\mu$ is an $f$-invariant measure, applying Oseledets Theorem we obtain:
  $$ \int \log |\det Df| \, d \mu = \int \sum_{k=1}^{\dim M} \chi_k  \, d \mu  =  \int \sum_{k=1}^{\dim  F} \chi_k  \, d \mu  + \int \sum_{k=\dim F + 1}^{\dim M} \chi_k  \, d \mu  = $$ $$\int \log|\det Df |_F | \, d \mu + \int \log|\det Df |_E | \, d \mu.$$
  Thus,
  \begin{equation}\label{eqn114}\int \log|\det Df |_F | \, d \mu = \int \log |\det Df| \, d \mu - \int \log|\det Df |_E | \, d \mu\end{equation}
   Besides, from standard inequalities of the linear algebra, and applying the definition of dominated splitting: $$\log |\det Df |_E| \leq \mbox{dim} (E) \log \|Df |_E \| \leq \mbox{dim}(E)\log \big(\sigma^{-1}  m(Df|_F) \big) \leq $$
     \begin{equation}\label{eqn115} -  \mbox{dim} (E) \,  \log \sigma  + \frac{\mbox{dim} (E)}{\mbox{dim} (F)} \, \log |\det Df |_F| . \end{equation}
   Joining  equality (\ref{eqn114}) with inequality (\ref{eqn115}):
   \begin{equation}\label{eqn116}  \Big(1 + \frac{\mbox{dim} (E)}{\mbox{dim} (F)}\Big) \, \int \log|\det Df |_F | \, d \mu \geq \int \log |\det Df| \, d \mu + {\mbox{dim} (E)} \log \sigma.\end{equation}
Finally, from inequalities (\ref{eqn110}), (\ref{eqn111}), (\ref{eqn112}) and (\ref{eqn116}) we conclude:
$$\Big(1 + \frac{\mbox{dim}(E)}{\mbox{dim} (F)}\Big) \,h_{\mu}(f) \geq \Big(1 + \frac{\mbox{dim}(E)}{\mbox{dim} (F)}\Big) \, \int \log|\det Df |_F | \geq \log |\det Df| \, d \mu + {\mbox{dim} (E)} \log \sigma   $$ $$ = s + \mbox{dim}(E) \log \sigma >  -\mbox{dim}(E) \log \sigma + \mbox{dim}(E) \log \sigma = 0.$$
We have proved that $h_{\mu}(f) >0$, as wanted.\end{proof}

\section{Proof  of Theorem \ref{Theorem3}.} \label{sectionProofSmooth}

The purpose of this section is to prove Theorem \ref{Smooth-Measure-Positive-Entropy}, completing all the proofs of   the results in the smooth invariant measure case (see Subsection \ref{subsectionStatementSmooth}). Recall Definition \ref{DefinitionAlmostDom} of almost-dominated splitting along an orbit and consider the following   result of \cite{SunTian}:

\begin{Thm}\label{PesFormula-Thm:positive}\noindent{\bf (Sun-T. \cite{SunTian})} Let $f     \in \mbox{Diff}^1(M)$  preserve  a smooth measure $\mu$ that has an almost dominated splitting. Then \em
\begin{equation}
\label{eqn120}
h_{\mu}(f)\geq\int  \sum_{i=1}^{\dim\,F}\chi_i(x)d\mu,\end{equation} \em where \em
 $\chi_1(x)\geq\chi_2(x)\geq\cdots\geq\chi_{\dim\,M}(x)$  \em are
the Lyapunov exponents at $ \mu$ a.e. $x.$
\end{Thm}

Note that $\mu$ is also a smooth invariant measure for   $f^{-1}.$ So,   Theorem \ref{PesFormula-Thm:positive} immediately implies \begin{equation}
\label{eqn120b}h_{\mu}(f)=h_{\mu}(f^{-1})\geq-\int  \sum_{i=1+ \dim\,F}^{\dim\,M}\chi_i(x)d\mu.\end{equation}

\vspace{.3cm}

\begin{proof}
{\em  of Theorem } \ref{Smooth-Measure-Positive-Entropy}:

By assumption, there is  a set $B$ with $\mu$ positive measure such that for any point $x\in B $ there exists a $N(x)-$dominated splitting $$T_{orb(x)}M=E_{orb(x)}\oplus F_{orb(x)}$$ such that $$0<\mbox{dim} (F(x))=i_0<\mbox{dim}(M),$$ where $N(\cdot):M\rightarrow
\mathbb{N}$ is an $f$-invariant measurable function and $i_0$ is a fixed integer number.

Let $B_L:=\{x\in B|\,\, N(x)\leq L\}.$ We have $B_L\subseteq B_{L+1}$ and $B=\bigcup_{L\geq 1}B_L.$ So, we can choose $L$ large enough such that $B_L$ has $\mu$ positive measure. Define $$S:=L ! \ \ \ \ \mbox{ and }\ \ \ \  g:=f^S.$$

Notice that $B_L$ is an $f-$invariant set, and besides  $N(x)| S $ for all   $x\in B_L$, Then, for any point $x\in B_L $
\begin{equation}
\label{eqn300}
\frac{\|Dg|_{E(x)}\|}{m(Dg|_{F(x)})}\leq \prod_{j=0}^{ ( {S}/{N(x)} ) -1}
\frac{\|Df^{N(x)}|_{E(f^{jN(x)}(x))}\|}{m(Df^{N(x)}|_{F(f^{jN(x)}(x))})}
\leq
\Big(\frac12 \Big )^{  {S}/{N(x)}}\leq\frac12 \ \ \ \forall \ x\in B_L.\end{equation}

Thus $g$ has a non trivial uniform dominated splitting for all $x \in B_L$. From inequality (\ref{eqn300}) we deduce the following, for all $x \in B_L$:
   \begin{eqnarray}
   \label{eqn230a}
    \limsup_{n\rightarrow +\infty}\frac1n\sum_{i=0}^{n-1}{\log \|Dg|_{E(g^{  i}(x))}\|}&\leq &\limsup_{n\rightarrow +\infty}\frac1n\sum_{i=0}^{n-1}{\log m(Dg|_{F(g^{  i}(x))})}-\log 2 \end{eqnarray}

Define $\nu :=\mu|_{B_L}$. Then $\nu$ is a $g$-invariant smooth measure.
From inequalities (\ref{eqn120}) and (\ref{eqn120b}) of Theorem \ref{PesFormula-Thm:positive} applied to $g$ in the role of $f$, we obtain $$\frac1 {i_0}h_{\nu}(g)\geq\int\chi_{i_0}(x)d\nu = \int \lim_{n \rightarrow + \infty} \frac{1}{n} \log m(Dg^n|_F) \, d\nu \geq $$ $$\int\limsup_{n\rightarrow +\infty}\frac1n\sum_{i=0}^{n-1}{\log m(Dg|_{F(g^{ i}(x))})}d\nu;$$
$$\frac1 { \mbox{dim}(M)-i_0} \cdot h_{\nu}(g)\geq-\int\chi_{i_0 + 1}(x)d\nu = -\int \lim_{n \rightarrow + \infty} \frac{1}{n} \log \|Dg^n)|_E \|\, d\nu\geq $$ $$-\int\liminf_{n\rightarrow +\infty}\frac1n\sum_{i=0}^{n-1}{\log \|Dg|_{E(g^{ i}(x))}\|}d\nu.$$

Taking the sum of the two latter inequalities, and applying inequality (\ref{eqn230a}), we conclude $$\Big(\frac1 {i_0}+\frac1 {\mbox{dim}(M)-i_0}\Big) \cdot h_{\nu}(g)\geq$$ $$\int\limsup_{n\rightarrow +\infty}\frac1n\sum_{i=0}^{n-1}{\log m(Dg|_{F(g^{ i}(x))})}\, d\nu-\int\liminf_{n\rightarrow +\infty}\frac1n\sum_{i=0}^{n-1}{\log \|Dg|_{E(g^{ i}(x))}\|}\, d\nu\geq   $$ $$\int\limsup_{n\rightarrow +\infty}\frac1n\sum_{i=0}^{n-1}{\log m(Dg|_{F(g^{ i}(x))})}\, d\nu-\int\limsup_{n\rightarrow +\infty}\frac1n\sum_{i=0}^{n-1}{\log \|Dg|_{E(g^{ i}(x))}\|}\, d\nu\geq \log 2>0.$$
So $h_{\topo}(f)=\frac1Sh_{\topo}(g)\geq \frac1Sh_{\nu}(g)>0,$
  as wanted.
\end{proof}

\section{ Proof  of Theorem \ref{Theorem4}} \label{sectionFinal}
The purpose of this section is to prove Theorem \ref{PartialHyper-SRB-like}. This Theorem explicits a positive lower bound for the entropy of  partially hyperbolic diffeomorphisms, and characterize a set of invariant measures whose metric entropy are bounded away from zero (see Subsection \ref{subsectionStatementSmooth}).

\vspace{.3cm}

\begin{proof}
{\em of Theorem} \ref{Theorem4}:

 Assertion (a) of Theorem \ref{PartialHyper-SRB-like}  assumes that    $F$ is an expanding sub-bundle. Precisely, if there exists $C>0$ and $\alpha >1$ such that $ {\|Df^{-n}|_{F(x)}\|} \leq C \alpha^{-n}$ for all $x\in M $ and for all $ n\geq 1$,     then
 $${m(Df^{n}|_{F(x)})}=\frac{1}{\|Df^{-n}|_{F(f^n(x))}\|}\geq C^{-1} \alpha^n.$$

 So, for any  regular point $x \in M$ and any vector $v$ in the Oseledets subspace $V  \subset F(x)$ with minimum Lyapunov exponent $\chi_{\dim (F)} (x)$ along $F$, we have:
  $$ \chi_{\dim (F)} (x)   =  \lim_{n \rightarrow \pm \infty} \frac{1}{n}  \, {\log \Big(\|Df^{n} v\|/\|v\|\Big)} \geq  \limsup_{n \rightarrow + \infty} \frac{1}{n} \, {\log \Big( m(Df^{n}|_{F(x)}\Big) }\geq  $$
     $$\chi_{\dim (F)} (x) \geq \lim_{n \rightarrow + \infty} \frac{1}{n} \log (C^{-1} \alpha^n) = \log \alpha >0.$$
  We have proved that for any regular point, the Lyapunov exponents along $F$   are bounded from below by $\log \alpha >0$.

  Thus, for any $f$- invariant probability measure $\mu$  we obtain:
  \begin{equation} \label{inequalityLemmaBorrado} \int \sum_{i= 1}^{\dim (F)} \chi_i(x) \, d \mu \geq \mbox{dim} (F) \int \chi_{\dim (F)}(x) \, d \mu \geq \mbox{dim}(F) \, \log \alpha. \end{equation}

To end the proof of   assertion (a)     recall (from Theorem \ref{TheoremCE})   that $p\omega(x,f) \subset \mathcal O_f$  for Lebesgue a.e. point $x \in M$. Take any $\mu \in p\omega(x,f)$. Joining Theorem \ref{Dom-EntropyFormula} with inequality (\ref{inequalityLemmaBorrado}), we conclude  $$h_{\mu}(f)\geq\int  \sum_{i=1}^{\dim\,F}\chi_i(x)d\mu= \geq  \mbox{dim}(F)\log \alpha,$$
as wanted. Finally, to prove assertion (b)    just apply assertion (a) replacing $f$ by $f^{-1}$ and $F$ by $E$. \end{proof}

\vspace{.3cm}

\section{Appendix}   \label{sectionExamples}
In this section we show two examples. The first   one, in Subsections \ref{subsectionExample0}, is a counterexample  that shows that the converse statement  of Theorems \ref{Dom-MainTheorem0} is false. The second example,  in Subsection \ref{subsectionExampleGourmelon-Potrie},   shows that the answer to Question \ref{Que} is negative. Namely, not all the diffeomorphisms  with (global and uniform) dominated splitting have positive topological entropy.

\subsection{Positive entropy with non-recurrent Lebesgue measure.} \label{sectionExample} \label{subsectionExample0}

 In this section we construct a simple example to show  that the hypothesis of $\delta$-recurrence of the Lebesgue measure  in Theorem \ref{Dom-MainTheorem0} (for some $0 < \delta < 1$) is  not necessary to have positive entropy.

\vspace{.3cm}

Consider the torus $\mathbb{T}^2$ and an area-preserving linear Anosov diffeomorphism $f_2: \mathbb{T}^2 \mapsto \mathbb{T}^2$ with expanding eigenvalue $\sigma_2>1$ and contracting eigenvalue $0 <\lambda_2 = \sigma_2^{-1} < 1$.
Denote by  $T \, \mathbb{T}^2= S \oplus U$  the hyperbolic splitting for $f_2$, where $S$ and $U$ are the stable and unstable sub-bundles respectively. Consider the circle $\mathbb{S}^1$ and a Morse-Smale order preserving diffeomorphism $f_1: \mathbb{S}^1 \mapsto \mathbb{S}^1$  having exactly two fixed points: a  hyperbolic sink $x_1$ and a  hyperbolic source $x_2$    such that $$0 <\lambda_1 := f'(x_1)= \min_{x \in \mathbb{S}^2} f'(x) < 1, \ \ \ 1 < \sigma_1 := f'(x_2) = \max_{x \in \mathbb{S}^2} f'(x) < \sigma_2.$$

Construct  $f: \mathbb{T}^3 = \mathbb{S}^1 \times \mathbb{T}^2 \mapsto \mathbb{T}^3 $  defined by  $ f(x,y) = (f_1(x), f_2(y)) \ \  \forall \ (x,y) \in \mathbb{S}^1  \times \mathbb{T}^2 .$
By construction $f$ has the following  $\sigma$-dominated splitting  $T \,  {\mathbb{T}^3}  = E \oplus F ,\  \mbox{ where } \  E = T \,   \mathbb{S}^1   \oplus S, \   F = U , \  \sigma = \sigma_2/\sigma_1> 1. $
Besides, $f$ has positive entropy. In fact, $f$ is the product map $f_1 \times f_2$, and so
 $h_{\topo}(f) = h_{\topo} (f_1) + h_{\topo}(f_2) = h_{\topo}(f_2) = \log \sigma_2 >0.$

Note that $f$ is non transitive, since it is a product map $f_1 \times f_2$ and $f_1$ is non transitive. So, this example shows that non transitive diffeomorphisms with dominated splitting may have positive entropy. Moreover, the wandering set is all the manifold except $\{x_0,x_1\} \times \mathbb{T}^2$. Thus, the wandering set has full Lebesgue measure, but even so, the entropy is positive.

For any $0 <\delta <1$ take in the circle $S^1$ two disjoint open neighborhoods of $x_1$ and $x_2$ such that the sum of their lengths is smaller than $\delta >0$. Denote by $K$ the compact set in $S^1$ that is the complement of the union of both open neighborhoods. Thus $$\Leb_{S^1}(K) > 1 - \delta, \ \ \Leb _{\mathbb{T}^3} (K \times \mathbb{T}^2) > 1- \epsilon.$$ Since $x_2$ is the omega-limit set in $S^1$ of all the orbits   by $f_1$, there exists $N \geq 1$ such that
 $f_1^n(K ) \cap K = \emptyset \ \ \forall \ n \geq N.$
Therefore,
 $f^n(K \times \mathbb{T}^2) \cap (K \times \mathbb{T}^2) = \emptyset \ \ \forall \ n \geq N.$
We conclude that the Lebesgue measure in $\mathbb{T}^3$ is non $\delta$-recurrent by $f$. Since $0 < \delta < 1$ can be arbitrarily chosen, this example shows  that diffeomorphisms with (global and uniform) dominated splitting and positive entropy exist such that the Lebesgue measure is non $\delta$-recurrent for all $0 < \delta < 1$. In brief, by Theorem \ref{Theorem1}, the $\delta$-recurrence condition of the Lebesgue measure is sufficient to ensure the positive entropy, but is is not a necessary condition.

\subsection{Gourmelon-Potrie example  \cite{Gourmelon-Potrie}}
 \label{subsectionExampleGourmelon-Potrie}
When searching for diffeomorphisms on surfaces, with dominated splitting and zero-entropy, one  must take into account two facts:

1) We are referring to Definition \ref{def:dominated} of a (global and uniform)   dominated splitting $TM= E \oplus F$, which must exist and \em be continuous \em at \em any   point \em of $TM$.

2)  The \em only surfaces \em admitting a non trivial global continuous  splitting are those that admit   a one-dimensional (never null)  sub-bundle of its tangent bundle,  \em continuously defined everywhere in $TM$. \em A classical theorem of   differential geometry states that those surfaces are only  the two torus $\mathbb{T}^2$ and the Klein bottle $\mathbb{K}^2$.

\vspace{.3cm}

  Considerations  1) and 2) impose strong restrictions that make  difficult the construction of  an example $f \in \mbox{Diff}^1_{DS}(M)$ with zero topological entropy, also in the case that $M$ has the lowest dimension (i.e $M$ is a surface, hence $M= \mathbb{T}^2$ or $M = \mathbb{K}^2$). Gourmelon-Potrie in \cite{Gourmelon-Potrie} have recently achieved such a construction  on $\mathbb{T}^2$, which we write here for the seek of completeness.

\noindent Take the covering $\pi: \mathbb{R}^2  \mapsto \mathbb{T}^2 = \big( \mathbb{R} / 2 \mathbb{Z}\big )^2$ defined by
 $  \pi(  x,  y) = \pi(  x',   y') \ \mbox{ if and only if } \  x- x'\in 2 \mathbb{Z}, \ \   y -   y' \in 2 \mathbb{Z}.$
For the seek of simplicity, we  will denote by $(x,y)$ a point in $\mathbb{R}^2$ and  also the point $\pi(x,y) \in \mathbb{T}^2$.
Fix two real numbers $a, b$ such that $0 < b < a < 1$, and
 define the following vector field in $\mathbb{T}^2$:
$$X(x,y) := \Big( \sin \pi x, \ \ \ a + b \cos \pi x    \Big).$$
Denote by $\phi(x,y,t)$ the tangent flow to $X$ (see Figure \ref{Figure1}). Namely,
 $\frac{d\phi}{dt} = X(\phi) \ \ \forall \ t \in \mathbb{R}, \ \ \ \phi(x,y,0) = (x,y).$
Define   the diffeomorphism $f: \mathbb{T}^2 \mapsto \mathbb{T}^2$ as the time 1 map of the flow $\phi$, i.e.
 $f(x,y) := \phi(x,y,1) \ \ \forall \ (x,y) \in \mathbb{T}^2.$ First, let us show that $h_{\topo}(f)= 0$, and second, let us construct a  global  uniformly  dominated splitting of for $f$.

\begin{Lem}
\label{lemma001}
 The   example $f \in \mbox{Diff}^1(\mathbb{T}^2)$ above constructed has null topological entropy.\end{Lem}

{\em Proof: } It is standard to check that the circles $S^1_0 = \{(x,y) \in \mathbb{T}^2: x= 0\}, \ \ \ \ S^1_{1} = \{(x,y) \in \mathbb{T}^2: x= 0\}$  are invariant by $f$. In fact,  $X(0,y) = (0,a + b), \ \phi(0, y, t) =\big(0, y + (a +b) t\big), \ \ f(0,y) = \big(0, y + (a +b) \big). $  So $f$ restricted to the circle $S_0^1$ is the rotation of angle $a + b >0$. Analogously
  $X(1,y) = (0,a-b), \ \phi(1, y, t) =\big(1, y + (a-b)t\big), \ \ f(1,y) = \big(1, y + (a-b)\big). $  So $f$ restricted to the circle $S_1^1$ is the rotation of angle $a-b >0$.

   Besides,   any orbit by $f$ that does not intersect the circles $S^1_0 \cap S_1^1$ has its $\alpha$-limit set contained in $S^1_0$ and its $\omega$-limit set contained in $S^1_1$ (see Figure \ref{Figure1}). This is because for any compact set in $\mathbb{T}^2 \cap \{0 < x < 1\}$ the horizontal component of the vector field $X$ is positive and bounded away from zero, and this horizontal component  changes its sign when applying the symmetry $(x,y) \mapsto (-x,y)$.

   Thus, the wandering set contains $\mathbb{T}^2  \setminus (S^1_0 \cup S^1_1)$. The recurrent points, and hence the support of all the invariant measures, are contained in $(S^1_0 \cup S^1_1)$. So, they are invariant measures by a rotation in the circle. Since the  entropy of   a rotation in the circle is null, we conclude that  the entropy of $f$ in the torus is also null. \hfill $\Box$

{\begin{figure}
[h]
\begin{center}
\vspace{0cm}
\includegraphics[scale=.4]{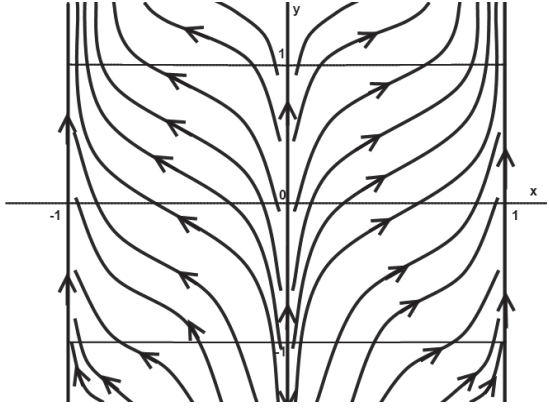}
\vspace{0cm}
\caption{\label{Figure1}  The flow $\phi$ tangent to the vector field $X$.}
\vspace{-.5cm}
\end{center}
\end{figure}}

\begin{Lem} \label{lemma002}
The Gourmelon-Potrie example $f \in \mbox{Diff}^1(\mathbb{T}^2)$ above constructed has a global uniformly dominated splitting.
\end{Lem}

\begin{proof}
Before constructing a dominated splitting for $f$, let us compute the derivative $Df$.
Since $\phi(x,y,t)$ is the solution of the differential equation $d\phi/dt = X(\phi)$ where $X   $ is a  $C^{1}$ vector field, we have:
 $\frac{d D\phi}{dt} = DX \cdot D\phi, \ \ \ D\phi = e^{{t \, DX} },$
where, for any $2 \times 2$ matrix $A$,    the exponential matrix $e^A$ is defined by
 $e^A = I + A + \frac{A^2}{2}  + \ldots + \frac{A^n}{n!} + \ldots$
For fixed $t= 1$ we obtain:
 $Df = e^{DX}, \mbox { where } DX = \left(
                                      \begin{array}{cc}
                                        \pi cos \pi x & 0 \\
                                        -b \pi \sin \pi x & 0 \\
                                      \end{array}
                                    \right),$
                                    and so
 \begin{equation}
 \label{eqnDf}
 Df = \left(
                                      \begin{array}{ll}
                                       \ \  e^{\pi cos \pi x} & 0 \\
                                        -e^{\pi cos \pi x} b \pi \sin \pi x & 1 \\
                                      \end{array}
                                    \right). \end{equation}
 Along the $f$-invariant circles $S^1_0$ and $S^1_1$, we  obtain:
 \begin{equation}
 \label{eqn81}
 Df_{(0,y)} = \displaystyle \left(
                                      \begin{array}{cc}
                                        e^{\pi} & 0 \\
                                        0 & 1 \\
                                      \end{array}
                                    \right), \ \ \ \ \
 Df_{(1,y)} = \displaystyle \left(
                                      \begin{array}{cc}
                                        e^{-\pi} & 0 \\
                                        0 & 1 \\
                                      \end{array}
                                    \right).\end{equation}
Thus, we can define the dominated splitting along   $S^1_0$ and $S^1_1$   as follows:
$$TM_{(0,y)} = E_{(0,y)} \oplus F_{(0,y)}, \ \ \ TM_{(1,y)} = E_{(1,y)} \oplus F_{(1,y)}, \mbox{ where } $$
\begin{equation}
\label{eqn82}
E_{(0,y)} = [(0,1)], \ \ F_{(0,y)} = [(1,0)], \ \ \ E_{(1,y)} = [(-1,0)], \ \ F_{(1,y)} =[(0,1)].\end{equation}
(The symbol $[v]$ denotes the subspace generated by $v$.)
To extend the dominated splitting to all the points of the wandering set, we will construct  $TM_{(x,y)} = E_{(x,y)} \oplus F_{(x,y)}$ for $0 < x < 1$, and then by symmetry of with respect to the axis $x= 0$, the splitting for $-1 < x < 0$ will be obtain as the symmetric of the splitting for $0 < x < 1$.

 Since $X(x,y)$, $D\phi(x,y,t)$ and $Df(x,y)$ are independent of $y$, we will define an splitting $  E_{(x,y)} \oplus F_{(x,y)}$  independent of $y$. We will construct it such that:

\vspace{.3cm}

  \noindent 1) When $x\rightarrow 0$, $E$ and $F$ converge  to $[(0,1)] $ and $ [(1,0)]$ respectively.

\noindent 2)  When $x\rightarrow 1$ $E$ and $F$ converge to $[(1,0)] $ and $ [(0,-1)]$ respectively.

\noindent 3) When $x$ increases in the interval $0 <x < 1$, $E$  and $F $ rotate  clockwise    angles  $0 < \psi_E(x), \psi_F(x)  < \pi/2$  that depend  continuously on $x$.

\noindent 4) The sub-bundles $E$ and $F$ are $Df$ invariant; i.e.

$Df(x,y)|_F(x,y) = F(f(x,y)), \ \ Df^{-1}(x,y)|_{E(x,y)} = E(f^{-1}(x,y)) \ \ \forall \ (x,y) \in \mathbb{T}^2.$

\noindent 5) The domination property holds. Namely, there exists    $C, \sigma >1$ such that
\begin{equation}
 \label{eqn80}
 |Df|_{E(x,y)}| \cdot  |Df^{-1}|_{F(f(x,y))}| < C \cdot \sigma ^{-1} \ \ \forall \ (x,y) \in \mathbb{T}^2. \end{equation}

Define  $\sigma:= e ^{\pi}/2 > 1$.
From equalities (\ref{eqn81}) and (\ref{eqn82})  along the invariant circles $S^1_0$ and $S^1_1$,  we obtain
 $|Df|_{F(0,y)}| = |Df(0,y)|_{[(1,0)]}| = e^{\pi}$ and $ \|Df|_{E(0,y)}| = |Df(0,y)|_{[(0,1)]}| = 1. $
 So
  $Df|_{E(0,y)} \cdot |Df^{-1}|_{F(f(0,y))}| =   e ^{-\pi } =  {\sigma^{-1}}/{2}.$
Analogously $|Df|_{E(1,y)}| = |Df(1,y)|_{[(1,0)]}| = e^{-\pi}$ and $  |Df|_{F(1,y)}| = |Df(1,y)|_{[(0,1)]}| = 1. $
 Hence \begin{equation}
 \label{eqn88_00}
  | Df|_{E(1,y)} | \cdot |Df^{-1}|_{F(f(1,y))}| =   e ^{-\pi } = \frac{ \sigma^{-1}}{2} <  \sigma^{-1}<1.\end{equation}
From the continuity of $Df$, there exists neighborhoods $U_0$ and $U_1$  of the circles $S^1_0$ and $S^1_1$ respectively, and   $\epsilon$-small open cones  $V_0$ and $V_1$ in the tangent bundles $TU_0$ and $TU_1$ respectively, containing the direction  $[(0,1)]$,   such that,
for all $(x,y) \in U_0 $ and for all $\ v_0 \in V_0$:
  \begin{equation}
  \label{eqn84}
  f^{-1}(x,y) \in U_0, \ \ \  [Df(x,y) v_0] \subset V_0   \ \ \mbox { and } \ \ \prec\big( Df(x,y) V_0  \big) < \sigma^{-1} \epsilon,   \end{equation} where $\prec\!\! V $ denotes the angle (not larger than $\pi$) of $V$, for any cone $V$ of directions in the tangent space $T_{(x,y)}\mathbb{T}^2 \equiv \mathbb{R}^2$.

    Analogously, replacing $f , E, F, S^1_0$ and $U_0, V_0$ by  $f^{-1}, F, E, S^1_1$ and $U_1, V_1$ respectively, we obtain the following assertions
for all $(x,y) \in U_1 $ and for all $\ v_1 \in V_1$:
     \begin{equation}
  \label{eqn86}f(x,y) \in U_1, \ \ \  [Df^{-1}(x,y) v_1] \subset V_1   \ \  \mbox { and } \ \ \prec\big( Df^{-1}(x,y) V_1  \big) < \sigma^{-1} \epsilon.\end{equation}

 Now, we will define a  continuous invariant sub-bundle $F$ in $TU_0$ and a continuous invariant    sub-bundle $E$ in $TU_1$, by the limits of the following   sub-bundles, which are uniformly convergent on $U_0$ and $U_1$ respectively  because $f^{-1}(U_0) \subset U_0, \ \ f(U_1) \subset U_1$ and due to  inequalities (\ref{eqn84}) and (\ref{eqn86}):
 \begin{equation}
 \label{eqn89a}
  \mbox{If } (x,y) \in U_0, \ \ \ \ F(x,y) := \lim_{n \rightarrow + \infty} Df^n(f^{-n}(x,y)) [(1,0)].
 \end{equation} \begin{equation}
 \label{eqn89b}
 \mbox{If } (x,y) \in U_1, \ \ \ \ E(x,y) := \lim_{n \rightarrow + \infty} Df^{-n}(f^{n}(x,y)) [(-1,0)].
 \end{equation}

By construction while $f^k(x,y) $ remains in $U_0$   the subspace $F_{f^k(x,y)}$ is $Df$-invariant. Analogously, while $f^{-k}(x,y)$ remains in $U_1$  the subspace $E_{f^{-k}}(x,y)$ is   $Df$-invariant. Besides, since $F$ and $E$ are continuous (on $U_0$ and $U_1$ respectively), and $[1,0]$ is invariant by $Df_{0,y}$ and $Df_{1,y}$, we deduce: $$\lim_{x \rightarrow 0} F(x,y) = F(0,y) = [(1,0)], \ \ \ \lim_{x \rightarrow 1} E(x,y) = E(1,y) = [(-1,0)].$$

Now let us extend continuously the invariant bundles $F$ and $E$ to the open subset   $  \mathbb{T}^2 \setminus (S^1_0 \cup S^1_1)$ in such a way that they remain $Df$-invariant.
For any $(x,y)$  such that $0 < x < 1$, there exists $N = N(x) \geq 1$ such that $f^{-N}(x,y) \in U_0$ and $f^{N}(x,y) \in U_1$. So, one can define:
\begin{equation}
\label{eqn87}
F(x,y) := Df^N(f^{-N}(x,y))\, F(f^{-N}(x,y)), \ \ \ \ E(x,y) := Df^{-N}(f^{ N}(x,y))\, E(f^{N}(x,y)). \end{equation}
Since $F$ and $E$ where previously defined to be $Df$-invariant and continuous in $U_0$ and $U_1$ respectively, their definition by equalities (\ref{eqn87}) in the   points $(x,y)$ such that $0 < x < 1$ does not depend on the choice of $N= N(x)$ (provided that $N(x)$ is large enough). So, $E$ and $F$ are $Df$-invariant and continuous in the open set $\{0 < x < 1\}$. Besides, $E$  and $F$   satisfy equalities (\ref{eqn89a}) and (\ref{eqn89b}) in the boundary $S^1_0 \cup S^1_1$ of that open set.
To prove that $E$ and $F$ are continuous $Df$-invariant sub-bundles in all the torus, it is left to prove the following equalities:

\begin{equation}
\label{eqn88} \lim_{x \rightarrow 1} F (x,y) = [(0,-1)] = F(1,y), \ \ \ \ \lim_{x \rightarrow 0} E (x,y) = [(0,1)] = E(0,y).
\end{equation}
Let us prove equalities (\ref{eqn88}). From equality (\ref{eqnDf}), for any fixed $ 0 < x < 1$ we have
 $\big[Df^n(x,y)  (1,0)\big] = \big[\big(1, -a_n(x)\big)\big] \ \
\mbox{ where } \     a_n(x) > 0 .$
 Thus, from equalities (\ref{eqn89a}) and (\ref{eqn89b}) we deduce   $F (x,y) = [1, -\alpha(x)] \mbox{ with } a =\lim_{n \rightarrow + \infty} a_n(f^{-n}(x)) \geq 0,  $ for all $(x,y) \in U_0 \setminus S^1_0$.
 But since $[1,0]$ is not $Df (x,y)$-invariant, if $0 < x < 1$, we obtain
  $F (x,y) = [1, -\alpha(x)] \ \ \mbox{ where  } \ \ \alpha(x) > 0 \ \ \forall \ (x,y) \in U_0 \setminus S^1_0.  $

 Now, we use equalities (\ref{eqn87}). We must apply    $Df^N(x,y)$ to the direction $[(1, -\alpha( x ))] $, where $Df$ is given by equality (\ref{eqnDf}), to obtain the direction $F $ at the point $f^N(x,y)$. Denote $\pi_1(x,y)= x, \ \ \pi_2(x,y)= y, \ \ x_1 = \pi_1(f(x,y)) , \ \ \ x_n = \pi_1(f^n(x,y)).$
  Denote   $$F(f(x,y)) = \big [Df(x,y) \cdot \big ( 1, \ -\alpha( x )\big)\big] = \big[\big( 1, \   -\alpha(x_1)   \big) \big].$$   A simple computation using formula (\ref{eqnDf}) gives  $ \alpha(x_1) =   b \pi \sin \pi x + e^{-\pi \cos \pi x } \alpha(x). $
  Thus, \begin{equation} \label{eqn90} \alpha(x_n) >0 \ \ \forall \ 0 < x < 1, \ \ \forall \ n \geq 0.\end{equation} If besides $x$ is sufficiently close to $1$ we have
   $\alpha(x_1) > 2 \alpha(x). $
Recall that $\alpha(x) >0$  and that $\lim_{N \rightarrow + \infty} \mbox{dist}(f^n(x,y), S_1^1)= 0$; namely $\lim_{n \rightarrow + \infty} x_n= 1$.   We deduce that for all $0 <x< 1$  there exists $N_0 = N_0(x)$ such that $\alpha(x_n) > 2^n \alpha(x)$ for all $n \geq N$. So, for any fixed $0 <x < 1$, $\lim_{n \rightarrow + \infty} \alpha(x_n) = + \infty.$

  Take any compact set $D \subset  {T^2} \cap \{0 < x < 1\}$ with non empty interior, such that any orbit  in the past with initial state in $U_1 \setminus S_1^1$ has at least one iterate  in $D$.   Choose any constant $K >0$. From the compactness of $  D$, there exists an  uniform $N \in \mathbb{N}$ such that
  $$\alpha(x_n) > K \ \ \forall \ n \geq N. $$ Construct the open set  $V_1 := \Big \{(x,y) \in U_1 \setminus S_1^1: \ \ f^{-j}(x,y) \not \in  D \ \ \forall \ j \in \{0,1, \ldots, N\} \Big\}.$  This open set $V_1$ is nonempty. In fact, arguing by contradiction, if it were $V_1 = \emptyset$, then the compact set $\bigcup_{j= 0}^N f^j(D)  $, which is at positive distance from $S_1^1$, would contain the open set $U_1 \setminus S_1^1$, which is at zero distance of $S^1_1$.

  By construction, for any point $(x,y) \in V_1$ there exists $n > N$ and $(x', y') = f^{-n}(x,y) \in D$. Thus $\alpha(x) > K$ for all $(x,y) \in V_1$. We have shown that for any constant $K$ there exists a neighborhood $V_1 \cup S_1$  of $S_1$ (with $V_1 \cap S_1 = \emptyset)$,  such that all the points $(x,y)$ in $V_1$ satisfy $\alpha(x) > K$. In other words, $$\lim_{x \rightarrow 1} \alpha(x) = + \infty.$$
  We conclude that
  $$\lim_{x \rightarrow 1} F(x,y) = \lim_{x \rightarrow 1}\Big[\Big(1, -\alpha(x) \Big) \Big] =\lim_{x \rightarrow 1}\Big[\Big (\frac{1}{\alpha(x)}, -1  \Big)\Big] = \Big[(0,1)\Big], $$
  as wanted. We have proved the equality at left in (\ref{eqn88}). To prove the equality at right, substitute $f, F, U_1, x= 1$ by $f^-1, E, U_0, x= 0$, and observe that $f^{-1}$ is the time 1 diffeomorphism of the flow tangent to the vector field $-X$. So $$Df^{-1} (x,y) = e^{-DX} =  \left(
                                      \begin{array}{ll}
                                        e^{-\pi cos \pi x} & 0 \\
                                        e^{-\pi cos \pi x} b \pi \sin \pi x & 1 \\
                                      \end{array}
                                    \right).       $$
  As in the above argument, we denote $x_{-n} = \pi_1f^{-n}(x,y)$ and prove that
        \begin{equation} \label{eqn91} E(x,y) = [(1, \beta(x))], \ \ \mbox{ where } \ \ \beta(x) >0 \ \ \forall \ 0 < x < 1.\end{equation} If besides $x$ is sufficiently close to $0$ we have
   $\beta(x_{-n}) > 2^n \beta(x),  \mbox{ and so } \lim_{x \rightarrow 0} \beta(x)= + \infty.$  We deduce that
  $$\lim_{x \rightarrow 0} E(x,y) = \lim_{x \rightarrow 0} \Big[ \Big( \frac{1}{ \beta(x)}, 1   \Big)  \Big] = \Big[ (0,1)\Big] = E(0,y), $$
  proving inequality at right in (\ref{eqn88}).

  \vspace{.3cm}

  Now, let us prove that $E \oplus F  = T \mathbb{T}^2$. In fact, by construction  both sub-bundles $E$ and $F$ are one-dimensional   and continuous. Since they are transversal along $S^1_0$ and $S^1_1$, they are still transversal in small open neighborhood   $U_0 \cup U_1$ of $S^1_0 \cup S^1_1$. Besides, from inequalities (\ref{eqn90}) and (\ref{eqn91})  they are uniformly transversal in the compact set $D= \{0 < x < 1\} \setminus (U_0 \cup U_1)$, because $-\alpha(x) < 0$ and $\beta(x) >0$, $E$ and $F$ are continuous and the set $D$ is compact. This proves that $E(x,y) \oplus F(x,y) = T_{(x,y)} M$ if $0 \leq x \leq 1$. Using a symmetry argument, one defines the invariant continuous splitting   $E \oplus F$ also on the points $\{-1 \leq x \leq 0\}$.

  Finally, we   prove that the $E \oplus F$ is an uniformly dominated splitting in the whole torus.

  On the one hand, from inequality (\ref{eqn88_00})   and due to the continuity of $Df$, $E$ and $F$, there exists a neighborhood $U$ of $S^1_0 \cup S^1_1$ such that
   $ |Df(x,y)  |_{E(x,y)} | \cdot |Df^{-1}(f(x,y))|_   {F(f(x,y))  }|  <\sigma^{-1} < 1 \ \ \forall \ (x,y) \in U. $ On the other hand, from the compactness of $\mathbb {T}^2 \setminus U$ there exists a constant $C > 1$ such that
   $ |Df(x,y)  |_{E(x,y)} | \cdot |Df^{-1}(f(x,y))|_   {F(f(x,y))  }|  <C \sigma^{-1}  \ \ \forall \ (x,y) \in \mathbb{T}^2 \setminus U. $

  We conclude that there exists $C, \sigma > 1$ such that

  $ |Df(x,y)  |_{E(x,y)} | \cdot |Df^{-1}(f(x,y))|_   {F(f(x,y))  }|  < C \sigma^{-1}  \ \ \forall \ (x,y) \in \mathbb{T}^2 , $
 as wanted. \end{proof}

\section*{Acknowlegements}    The authors thank Rafael Potrie and Jos\'{e} Vieitez  for the rich discussions and suggestions to improve this paper, and  N. Gourmelon and R. Potrie for providing  their example \cite{Gourmelon-Potrie}.


\end{document}